\documentclass[10pt,a4paper]{amsart}
\usepackage{amssymb,amsmath,amsfonts}
\usepackage[colorlinks=true,linkcolor=blue,citecolor=blue,urlcolor=black,bookmarksopen=true]{hyperref}
\usepackage{bookmark}
\headsep=1.2500cm \numberwithin{equation}{section}
\newcommand{\norm}[1]{\left\Vert#1\right\Vert}
\newcommand{\abs}[1]{\left\vert#1\right\vert}

\newtheorem{Theorem}{Theorem}[section]
\newtheorem{Proposition}[Theorem]{Proposition}
\newtheorem{cor}[Theorem]{Corollary}
\newtheorem{lemma}[Theorem]{Lemma}
\theoremstyle{remark}

\newtheorem{Example}[Theorem]{Example}

\newtheorem{Remark}[Theorem]{Remark}

\begin{document}

\title[Johnson pseudo-contractibility of certain Banach algebras]{Johnson pseudo-contractibility of certain Banach algebras and their nilpotent ideals}

\author[M. Askari-Sayah]{M. Askari-Sayah}

\author[A. Pourabbas]{A. Pourabbas}

\address{Faculty of Mathematics and Computer Science,
Amirkabir University of Technology, 424 Hafez Avenue, 15914 Tehran,
 Iran.}
\email{mehdi17@aut.ac.ir}

\email{arpabbas@aut.ac.ir}

\author[A. Sahami]{A. Sahami}

\address{Department of Mathematics,
Faculty of Basic Sciences, Ilam University, P.O. Box 69315-516, Ilam,
Iran.}

\email{a.sahami@ilam.ac.ir}



\maketitle

\begin{abstract}
In this paper, we study the notion of Johnson pseudo-contractibility
for certain Banach algebras. For a  bicyclic semigroup $S$, we show that $\ell^{1}(S)$ is not Johnson pseudo-contractible.
Also for a Johnson pseudo-contractible Banach
algebra $A$, we show that $A$ has  no non-zero   complemented  closed nilpotent
ideal.\end{abstract}
\section{Introduction}
The theory of amenability for Banach algebras was  developed by  Johnson \cite{john} in 1972. This theory has been occupying an important place in research in modern analysis. But amenability for Banach algebras remains a too strong concept. This fact leads some theorists to introduce other concepts by relaxing some different conditions in the definition of amenable Banach algebras. F. Ghahramani and R. Loy \cite{generalize1} defined the concept of approximate amenability for Banach algebras. Pseudo-amenability and pseudo-contractibility, two important generalizations of amenability for Banach algebras, were introduced by F. Ghahramani and Y. Zhang in \cite{ghah-pse}. These notions have been studied for various classes of Banach algebras we may mention, for example, semigroup algebras, Segal algebras and Fourier algebras.

 Recently,  the second and third named authors defined
a new concept of amenability called Johnson pseudo-contractibility \cite{sah1}.
Johnson pseudo-contractibility of some certain Banach algebras like
group algebras, measure algebras and Lipschitz algebras were studied
in \cite{sah1}. Furthermore, the second and third authors in \cite[Example
4.1-(iii)]{sah1} showed that $M_{\mathbb{N}}(\mathbb{C})$ (The Banach algebra of all $ \mathbb{N}\times\mathbb{N} $-matrices $ (z_{ij}) $  with entries in $ \mathbb{C} $ such that $ \norm{(z_{ij})}=\sum\limits_{i,j}\abs{z_{ij}}<\infty $, $ \ell^{1} $-norm and matrix multiplication) is not Johnson
pseudo-contractible but $M_{\mathbb{N}}(\mathbb{C})$ is
pseudo-amenable.

The existence of  non-zero nilpotent ideals in amenable Banach algebras have been investigated by many authors. Firstly, In 1989, P. Curtis and R. Loy \cite{curt} showed that non-zero nilpotent ideals in amenable Banach algebras must be infinite dimensional. Also, nilpotent ideals  which have approximation property
 in amenable Banach algebras was studied by R. Loy and G. Willis \cite{loy} in 1994. Inspired by these works, Y. Zhang defined the concept of approximately complemented subspaces of Banach spaces. Moreover, he showed that a pseudo-contractible Banach algebra does not have a non-trivial closed nilpotent ideal which is approximately complemented. For further details see \cite{zhang}.

 The first purpose of this paper is to study Johnson pseudo-contractibility of certain Banach algebras including semigroup algebras and Segal algebras. The second purpose is to study nilpotent ideals in Johnson pseudo-contractible Banach algebras.

 In section \ref{sec1} we show
that for a unital Banach algebra $A$, Johnson pseudo-contractibility
is equivalent with amenability. This allows us to characterize the Johnson pseudo-contractibility of a semigroup algebra provided that the semigroup has an identity. As an application we show that for $\mathbb{N}$ with maximum as its
product, the semigroup algebra is never Johnson
pseudo-contractible (but it is pseudo-amenable). Again using this
criteria we can see that $\ell^{1}(S)$ fails to be Johnson
pseudo-contractible, whenever $S$ is a bicyclic semigroup.

In section
\ref{sec2} we define a new homological notion for Banach algebras called property $ \digamma $. This property has an important role in section \ref{sec3}. Some details of property $ \digamma $ will be given and some relation between property $ \digamma $ and Johnson pseudo-contractibility will be studied. Also, we show that in particular cases Johnson pseudo-contractibility of $ S^{1}(G) $ is equivalent with amenability of $ G $.

 Finally, in section \ref{sec3}, we turn our attention to study the structure of Johnson pseudo-contractible Banach algebras. In fact, we show
that a Johnson pseudo-contractible Banach algebra $A$ has no
non-zero closed nilpotent ideal which is  complemented
in $A.$

We recall some standard notation and definitions.
Suppose that $ X $ and $ Y $ are normed spaces. The projective tensor
product of $ X $ and $ Y $ is denoted by $ X\otimes_{p} Y$, and the space of all bounded linear maps from $ X $ into $ Y $ is denoted by $ B(X, Y) $.
Suppose that $ X_{1} $, $ X_{2} $, $ Y_{1} $ and $ Y_{2} $ are normed spaces and $ T_{i}\in B(X_{i}, Y_{i}) $, $ i = 1, 2 $. The
tensor product of $ T_{1} $ and $ T_{2} $ is the bounded linear map $ T_{1}\otimes T_{2} $, determined by
\begin{equation*}
(T_{1}\otimes T_{2})(x_{1}\otimes x_{2} ) = T_{1}(x_{1})\otimes T_{2}(x_{2}) \qquad (x_{1}\in X_{1},  x_{2}\in X_{2} ).
\end{equation*}

For a Banach algebra $ A $ the product map on $ A $ determines  a map $ \pi_{A}:A\otimes_{p}A\rightarrow A $, specified by $ \pi_{A}(a\otimes b)=ab $ for all $ a, b\in A $. The projective tensor product $ A\otimes_{p}A $
becomes a Banach $ A$-bimodule with the following module actions:
\begin{equation*}
a\cdot (b \otimes c) = ab \otimes c,\quad (b \otimes c)\cdot a = b \otimes ca \qquad(a, b, c \in A),
\end{equation*}
 and with these actions $ \pi_{A} $ becomes an $ A $-bimodule morphism.
  We regard the dual space  $ A^*$ as a Banach $A$-module with the operations defined by $(af)(b)=f(ba),\quad (fa)(b)=f(ab)$ for all $a,b\in A$ and $f\in A^*$. The first Arens product on the the second dual $  A^{\ast\ast} $ of $A$ is defined by
 $(F\square H)(f)=F(Hf),$  where $(Hf)(a)=H(fa)$ for all $F,H\in A^{**},\, f\in A^*$ and $a\in A$.

Suppose that $A$ is a Banach algebra. $A$ is said to be amenable if it has a virtual diagonal, that is, there
exists an element $M\in (A\otimes_{p}A)^{**}$ such that $a\cdot M=M\cdot a$ and
$\pi^{**}_{A}(M)a=a$ for every $a\in A$.  $A$ is pseudo-amenable (pseudo-contractible) if there
exists a net $(m_{\alpha})$ in
$A\otimes_{p}A$ such that $a\cdot m_{\alpha}-m_{\alpha}\cdot
a\rightarrow 0\,$ $(a\cdot m_{\alpha}=m_{\alpha}\cdot a)$ and
$\pi_{A}(m_{\alpha})a-a\rightarrow 0,$ for every $a\in A,$
respectively.
Moreover, $ A $ is called approximately amenable if and only if $ A^{\sharp} $, the unitization of $ A $, is pseudo-amenable, and
is called Johnson pseudo-contractible if there exists a net
$(m_{\alpha})$ in $(A\otimes_{p}A)^{**}$ such that $a\cdot
m_{\alpha}=m_{\alpha}\cdot a$ and
$\pi^{**}_{A}(m_{\alpha})a-a\rightarrow 0,$ for every $a\in A.$ Note that
 Johnson pseudo-contractibility is strictly weaker than pseudo-contractibility, see \cite[Example 4.1-(ii)]{sah1}.

There are some equivalent definitions of amenable and approximately amenable Banach algebras, see \cite{generalize1, generalize2, john} and \cite{run}.
\section{Johnson pseudo-contractibility of certain semigroup algebras}\label{sec1}

\begin{Theorem}\label{jons imply amen}
Let $ A $ be a Johnson pseudo-contractible Banach algebra with an
identity. Then $ A $ is amenable.
\end{Theorem}
\begin{proof}
By hypothesis, let $ (m_{\alpha}) $ be a net in
$(A\otimes_{p}A)^{**}$ such that $a\cdot m_{\alpha}=m_{\alpha}\cdot
a$ and $\pi^{**}_{A}(m_{\alpha})a-a\rightarrow 0$ for every $a\in
A.$
 By \cite[Lemma 1.7]{ghah}, there exists a bounded linear map $
\psi:A^{\ast\ast}\otimes_{p}{A}^{\ast\ast}\longrightarrow
(A\otimes_{p} A)^{\ast\ast} $
that satisfies
\begin{enumerate}
\item[(i)]$ \psi(a\otimes b)=a\otimes b, $
\item[(ii)]$ a\cdot\psi(m)=\psi(a\cdot m),\hspace*{1cm} \psi(m)\cdot a=\psi(m\cdot a), $
\item[(iii)]$ \pi_{A}^{\ast\ast}(\psi(m))=\pi_{A^{\ast\ast}}(m), $
\end{enumerate}
for every $ a,b\in A $ and $ m\in
A^{\ast\ast}\otimes_{p}{A}^{\ast\ast}$.
 Set $
\omega_{\alpha}=\psi(\pi_{A}^{\ast\ast}(m_{\alpha})\otimes
e)-m_{\alpha} $, where $e\in A$ is the identity element. We have
\begin{equation*}
\pi_{A}^{\ast\ast}(\omega_{\alpha})=\pi_{A}^{\ast\ast}(\psi(\pi_{A}^{\ast\ast}(m_{\alpha})\otimes
e)-m_{\alpha})=\pi_{A^{\ast\ast}}(\pi_{A}^{\ast\ast}(m_{\alpha})\otimes
e)-\pi_{A}^{\ast\ast}(m_{\alpha})=\pi_{A}^{\ast\ast}(m_{\alpha})-\pi_{A}^{\ast\ast}(m_{\alpha})=0.
\end{equation*}
So $ \omega_{\alpha}\in\ker\pi_{A}^{\ast\ast} $. Since $ A $ is
unital, $ \pi_{A} $ is surjective. Thus \cite[A.3.48]{auto-dale}
implies that
\begin{equation}\label{nvhgbh}
\overline{\ker\pi_{A}}^{wk^{\ast}}=(\ker\pi_{A})^{\ast\ast}=\ker\pi_{A}^{\ast\ast}.
\end{equation}
Consider  $ \ker\pi_{A} $ as a  closed left ideal of $
A\otimes_{p}A^{op} $, where $ A^{op}  $ is reversed algebra of $ A $, and let
 $ v=\sum\limits_{i}x_{i}\otimes y_{i}\in \ker\pi_{A} $. We have
\begin{equation*}\begin{split}
v\cdot\omega_{\alpha} =& \sum\limits_{i}x_{i}\otimes
y_{i}\cdot(\psi(\pi_{A}^{\ast\ast}(m_{\alpha})\otimes e)-m_{\alpha})
\\ =& \sum\limits_{i}x_{i}\otimes
y_{i}\cdot\psi(\pi_{A}^{\ast\ast}(m_{\alpha})\otimes
e)-\sum\limits_{i}x_{i}\cdot m_{\alpha}\cdot y_{i}\\=&
\sum\limits_{i}x_{i}\otimes
y_{i}\cdot\psi(\pi_{A}^{\ast\ast}(m_{\alpha})\otimes
e)-m_{\alpha}\cdot\sum\limits_{i}x_{i} y_{i}\\=& v\cdot
\psi(\pi_{A}^{\ast\ast}(m_{\alpha})\otimes e).
\end{split}\end{equation*}
Given $ V\in \ker\pi_{A}^{\ast\ast}$, by \eqref{nvhgbh} we can
choose $ (v_{\beta})_{\beta}\subseteq\ker\pi_{A} $ such that $
w^{\ast}-\lim\limits_{\beta}v_{\beta}=V $, and by the above calculations, we
have
\begin{equation*}
\begin{split}
 V \square\omega_{\alpha}=&(w^{\ast}-\lim\limits_{\beta}v_{\beta})\square \omega_{\alpha}\\=&
 w^{\ast}-\lim\limits_{\beta}(v_{\beta}\cdot \omega_{\alpha})\\=&
 w^{\ast}-\lim\limits_{\beta}(v_{\beta}\cdot \psi(\pi_{A}^{\ast\ast}(m_{\alpha})\otimes e))\\=&
 (w^{\ast}-\lim\limits_{\beta}v_{\beta})\square \psi(\pi_{A}^{\ast\ast}(m_{\alpha})\otimes e)\\=&
 V\square\psi(\pi_{A}^{\ast\ast}(m_{\alpha})\otimes e).
\end{split}
\end{equation*}
Moreover,  $ \pi_{A}^{\ast\ast}(m_{\alpha})\longrightarrow e $
implies that
\begin{equation*}
\lim\limits_{\alpha}\psi(\pi_{A}^{\ast\ast}(m_{\alpha})\otimes
e)=\psi(\lim\limits_{\alpha}\pi_{A}^{\ast\ast}(m_{\alpha})\otimes
e)=\psi(e\otimes e)=e\otimes e,
\end{equation*}
hence
\begin{equation*}
\begin{split}
\lim\limits_{\alpha}\sup\limits_{V}\norm{V
\square\omega_{\alpha}-V}=&\lim\limits_{\alpha}\sup\limits_{V}
\norm{V \square \psi(\pi_{A}^{\ast\ast}(m_{\alpha})\otimes
e)-V}\\\leq&\lim\limits_{\alpha}\sup\limits_{V}\norm{V}\norm{\psi(\pi_{A}^{\ast\ast}(m_{\alpha})\otimes
e)-e\otimes e}=0,
\end{split}
\end{equation*}
where supremum is taken on all $ V\in{\hbox{ball}} (\ker\pi_{A}^{\ast\ast})= \{x\in
\ker\pi^{**}_{A}:||x||\leq 1\}$.
 Therefore $ R_{\omega_{\alpha}} $,
the right multiplication operator, converges to $
\mathrm{id}_{\ker\pi_{A}^{\ast\ast}} $ in the norm topology of $
{\hbox{ball}}(\ker\pi_{A}^{\ast\ast}) $. Thus, there is $ \alpha_{0}
$ such that for every $ \alpha\geq\alpha_{0} $, $R_{\omega_{\alpha}}
$ is invertible.  Since $ R_{\omega_{\alpha}} $ is surjective, for
some $ \Upsilon\in\ker\pi_{A}^{\ast\ast} $ we have  $
\Upsilon\square\omega_{\alpha}=\omega_{\alpha}.$  Given $
\Lambda\in\ker\pi_{A}^{\ast\ast} $, we have $
(\Lambda\square\Upsilon-\Lambda)\square\omega_{\alpha}=\Lambda\square\Upsilon\square\omega_{\alpha}-\Lambda\square\omega_{\alpha}=0
$. By injectivity of $ R_{\omega_{\alpha}} $ it follows that $
\Lambda\square\Upsilon=\Lambda $. So that $ \Upsilon $ is a right
identity for $ \ker\pi_{A}^{\ast\ast} $. Let $ \Phi=e\otimes
e-\Upsilon $. Then for each $ a\in A $ we have
\begin{equation*}\begin{split}
a\cdot\Phi-\Phi\cdot a&=a\cdot(e\otimes e-\Upsilon)-(e\otimes
e-\Upsilon)\cdot a\\&= a\otimes e-a\cdot\Upsilon-e\otimes
a+\Upsilon\cdot a\\&= a\otimes e-(a\otimes e)\Upsilon-e\otimes
a+(e\otimes a)\Upsilon\\&= (a\otimes e-e\otimes a)(e\otimes
e-\Upsilon)=0,
\end{split}\end{equation*}
and
\begin{equation*}
\pi_{A}^{\ast\ast}(\Phi)=\pi_{A}^{\ast\ast}(e\otimes
e)-\pi_{A}^{\ast\ast}(\Upsilon)=e,
\end{equation*}
which implies  that $ \Phi $ is a virtual diagonal for $ A $.
Therefore $A$ is amenable.
\end{proof}
\begin{Example}
Let $S$ be a bicyclic semigroup, that is, a semigroup generated by
two elements $p$ and $q$ such that $pq=e\neq qp$, where $e$ is the identity for
$S$, see \cite[Example 2.10]{ber}. It is well-known that  $\ell^{1}(S)$ is not amenable \cite[Remarks (1)]{dun1978}. But since
$\ell^{1}(S)$ is unital, Theorem \ref{jons imply amen} implies that
$\ell^{1}(S)$ is not Johnson pseudo-contractible.
\end{Example}
\begin{Remark}
It is well-known that a Banach  algebra  $A$ is amenable if and only
if $A^\sharp$ (unitization of $A$) is amenable, see \cite[Corollary
2.3.11]{run}. This property is not valid  in Johnson
pseudo-contractibility case. To see this let $S=\mathbb{N}\cup\{0\}$.
With the following action
\begin{eqnarray*}
&\textit{$m\ast
n=$}\begin{cases}m\,\,\,\,\,\,\,\,\,\,\,\,\,\,\,\,\,\ if\qquad
m=n\cr  0\,\,\,\,\,\,\,\,\,\,\,\,\,\,\,\,\,\,\,\,if\qquad m\neq n,
\end{cases}\\
\end{eqnarray*}   $S$ becomes a
semigroup. By \cite[Corollary 2.7]{rost1} $\ell^{1}(S)$
is pseudo-contractible. Thus by \cite[Lemma 2.2]{sah1} $\ell^{1}(S)$
is Johnson pseudo-contractible. We claim that $\ell^{1}(S)^\sharp$
is not Johnson pseudo-contractible. Suppose conversely that
$\ell^{1}(S)^\sharp$ is Johnson pseudo-contractible. Then by Theorem
\ref{jons imply amen},  $\ell^{1}(S)^\sharp$ is amenable. So
$\ell^{1}(S)$ is amenable and \cite [Theorem 2]{ddun-pat} implies that the
set of idempotents of $S$ is  finite, a contradiction.
\end{Remark}

A semigroup $S$ is called regular if for each $s\in S$ there exists
an element $t\in S$ such that $sts=s$ and $tst=t$. The set of idempotents
of semigroup $S$ is  denoted by $E_{S}.$
\begin{cor}\label{cor regular}
Let $S$ be a semigroup with identity. If $\ell^{1}(S) $ is
Johnson-pseudo-contractible, then $S$ is regular and $E_{S}$ is
finite.
\end{cor}
\begin{proof}
Let $S$ be a semigroup with identity. Then $\ell^{1}(S) $ is unital.
Using Theorem \ref{jons imply amen}, Johnson pseudo-contractibility
of  $\ell^{1}(S) $ implies that $\ell^{1}(S)$ is amenable. Applying
\cite[Theorem 2]{ddun-pat}, $S$ must be regular and $E_{S}$ is
finite.
\end{proof}
\begin{Theorem}\label{dual}
Let $ A $ be a  Banach algebra and let $ A^{\ast\ast} $ be Johnson
pseudo-contractible. Then $ A $ is Johnson pseudo-contractible.
\end{Theorem}
\begin{proof}
Since  $ A^{\ast\ast} $ is Johnson pseudo-contractible, there exists
a net $( M_{\alpha})\subseteq
(A^{\ast\ast}\otimes_{p}{A}^{\ast\ast})^{\ast\ast} $ such that $
F\cdot M_{\alpha}=M_{\alpha}\cdot F $ and $
F\cdot\pi_{A^{\ast\ast}}^{\ast\ast}(M_{\alpha})\longrightarrow F $
for every $ F\in A^{\ast\ast} $. In particular for $ a\in A $ we
have
\begin{equation*}
a\cdot M_{\alpha}=M_{\alpha}\cdot a \qquad and\qquad
a\cdot\pi_{A^{\ast\ast}}^{\ast\ast}(M_{\alpha})\longrightarrow a.
\end{equation*}
Let $ \psi:A^{\ast\ast}\otimes_{p}{A}^{\ast\ast}\longrightarrow
(A\otimes_{p} A)^{\ast\ast} $ be as in the proof of  Theorem \ref{jons imply amen} and let
 $ \theta: (A\otimes_{p}{A})^{\ast}\longrightarrow
(A\otimes_{p}{A})^{\ast\ast\ast}$ be the canonical embedding map.
Both $ \psi^{\ast\ast} $ and $ \theta^{\ast} $ are weak*
continuous $ A$-bimodule morphisms. Set $
N_{\alpha}=\theta^{\ast}\psi^{\ast\ast}(M_{\alpha})\in
(A\otimes_{p} A)^{\ast\ast}  $. For each $a\in A$ we have
\begin{equation*}
a\cdot
N_{\alpha}=a\cdot\theta^{\ast}\psi^{\ast\ast}(M_{\alpha})=\theta^{\ast}\psi^{\ast\ast}(a\cdot
M_{\alpha})=\theta^{\ast}\psi^{\ast\ast}(M_{\alpha}\cdot
a)=\theta^{\ast}\psi^{\ast\ast}(M_{\alpha})\cdot a=N_{\alpha}\cdot
a.
\end{equation*}
Now for every $ \alpha $ we choose ($ x_{\beta})\subseteq
A^{\ast\ast}\otimes_{p}{A}^{\ast\ast} $ such that $
w^{\ast}-\lim\limits_{\beta} x_{\beta}=M_{\alpha} $, for each $ a $ in $ A $ we have
\begin{equation*}
\begin{split}
 a\cdot \pi_{A}^{\ast\ast}(N_{\alpha})=&a\cdot \pi_{A}^{\ast\ast}(\theta^{\ast}\psi^{\ast\ast}(M_{\alpha}))
 \\=& a\cdot w^{\ast}-\lim\limits_{\beta}\pi_{A}^{\ast\ast}(\theta^{\ast}\psi^{\ast\ast}(\hat{x}_{\beta}))
 \\=& a\cdot w^{\ast}-\lim\limits_{\beta}\pi_{A}^{\ast\ast}(\theta^{\ast}(\widetilde{\psi(x_{\beta})}))
 \\=& a\cdot w^{\ast}-\lim\limits_{\beta}\pi_{A}^{\ast\ast}(\psi(x_{\beta}))
 \\=& a\cdot w^{\ast}-\lim\limits_{\beta}\pi_{A^{\ast\ast}}^{\ast\ast}(x_{\beta})
 \\=& a\cdot\pi_{A^{\ast\ast}}^{\ast\ast}(M_{\alpha}),
\end{split}
\end{equation*}
where $ \widehat{\cdot}:A^{\ast\ast}\otimes_{p}{A}^{\ast\ast}\to (A^{\ast\ast}\otimes_{p}{A}^{\ast\ast})^{**} $ and $ \widetilde{\cdot}:(A\otimes_{p}{A})^{\ast\ast}\to (A\otimes_{p}{A})^{\ast\ast\ast\ast}$ are the canonical embedding maps, respectively.  Thus $ a\cdot
\pi_{A}^{\ast\ast}(N_{\alpha})=a\cdot\pi_{A^{\ast\ast}}^{\ast\ast}(M_{\alpha})\longrightarrow
a $. It follows that $A$ is Johnson pseudo-contractible.
\end{proof}
\begin{Example}
Consider the semigroup $\mathbb{N}_{\vee}$ with the semigroup
multiplication $m\vee n=\max\{m,n\}$ for each $m,n\in \mathbb{N}.$
We claim that $\ell^{1}(\mathbb{N}_{\vee})$ is not Johnson
pseudo-contractible. Since $\ell^{1}(\mathbb{N}_{\vee})$ is unital, if $\ell^{1}(\mathbb{N}_{\vee})$
is Johnson pseudo-contractible, then by Corollary \ref{cor regular} $E_{\mathbb{N}_{\vee}}$
is finite which is a contradiction. Furthermore by \cite[Example
4.6]{generalize2} we know that $\ell^{1}(\mathbb{N}_{\vee})$ is
approximately amenable. Since $\ell^{1}(\mathbb{N}_{\vee})$ is
unital, \cite[Proposition 3.2]{ghah-pse} implies that
$\ell^{1}(\mathbb{N}_{\vee})$ is pseudo-amenable. Therefore we have
a Banach algebra among semigroup algebras which is not Johnson
pseudo-contractible but it is pseudo-amenable and
approximately amenable.
Moreover, by Theorem \ref{dual} $\ell^{1}(\mathbb{N}_{\vee})^{**}$ is not
Johnson pseudo-contractible.
\end{Example}
\section{Property $\digamma$ and Johnson pseudo-contractibility}\label{sec2}

A Banach algebra $A$ has  property $\digamma$ if there exists a not necessary bounded net  $(\rho_{\alpha})_{\alpha\in I}\subseteq B(A,(A\otimes_{p}A)^{**})$ of $A$-bimodule morphisms such that
$$\pi_{A}^{**}\circ\rho_{\alpha}(a)-a\rightarrow 0,\quad (a\in A).$$
\begin{Proposition}\label{F property}
Let $A$ be a Banach algebra with a central approximate identity. Then $A$
is Johnson pseudo-contractible if and only if $A$ has  property
$\digamma$.
\end{Proposition}
\begin{proof}
Suppose that $A$ is Johnson pseudo-contractible. Then there exists a net
$(m_{\alpha})_{\alpha \in I}$ in $(A\otimes_{p}A)^{**}$ such that
$a\cdot m_{\alpha}=m_{\alpha}\cdot a$ and
$\pi^{**}_{A}(m_{\alpha})a-a\rightarrow 0$ for every $a\in A.$
Define $\rho_{\alpha}:A\rightarrow (A\otimes_{p}A)^{**}$ by
$\rho_{\alpha}(a)=m_{\alpha}\cdot a $ for every $a\in A.$ Clearly
$(\rho_{\alpha})$ is a net of $A$-bimodule morphisms in $ B(A,(A\otimes_{p}A)^{**}) $ such that
$$\pi_{A}^{**}\circ\rho_{\alpha}(a)-a=\pi^{**}_{A}(m_{\alpha}\cdot a)-a=\pi^{**}_{A}(m_{\alpha})a-a\rightarrow 0\quad (a\in A).$$

Conversely, suppose that $A$ has  property $\digamma$. Then there
exists  a net $(\rho_{\alpha})_{\alpha\in I}\subseteq B(A,(A\otimes_{p}A)^{**})$ such that
$$\pi_{A}^{**}\circ\rho_{\alpha}(a)-a\rightarrow 0,\quad (a\in A).$$ Let $(e_{\beta})_{\beta \in J}$ be a central approximate
identity. Define $(m^{\beta}_{\alpha})_{(\alpha,\beta)\in (I,J)}$ by
$m^{\beta}_{\alpha}=\rho_{\alpha}(e_{\beta})$. It is easy to see
that $$a\cdot m^{\beta}_{\alpha}-m^{\beta}_{\alpha}\cdot
a=a\cdot\rho_{\alpha}(e_{\beta})-\rho_{\alpha}(e_{\beta})\cdot
a=\rho_{\alpha}(ae_{\beta})-\rho_{\alpha}(e_{\beta}a)=0 $$ and
$$\lim_{\beta}\lim_{\alpha}\pi^{**}_{A}(m^{\beta}_{\alpha})a-a=\lim_{\beta}\lim_{\alpha}\pi^{**}_{A}\circ\rho_{\alpha}(e_{\beta})a-a=\lim_{\beta}e_{\beta}a-a=0,$$
for every $a\in A.$ Using iterated limit theorem \cite[page
69]{kel}, we can find a net $n_{\gamma}\in (A\otimes_{p}A)^{**}$
such that
$$a\cdot n_{\gamma}=n_{\gamma}\cdot a,\quad
\pi^{**}_{A}(n_{\gamma})a-a\rightarrow 0\qquad (a\in A).$$ So $A$ is
Johnson pseudo-contractible.
\end{proof}
Note that the existence of central approximate identity in the ``if part'' of previous proposition is not necessary.
\begin{Remark}\label{rem wot}
We recall that if $E$ and $F$ are Banach spaces, then the weak$^{*}$ operator topology on $B(E,F^{*})$ is the locally convex topology determined by the seminorms  $\{p_{e,f}:e\in E, f\in F^{*}\}$, where $p_{e,f}(T)=|\langle f,Te\rangle|$. Moreover closed balls of $B(E,F^*)$ are compact with respect to the  weak$^{*}$ operator topology (see \cite[Theorem A.3.35]{auto-dale}).
\end{Remark}
\begin{Remark}
Note that in the definition of property $\digamma$ if $
(\rho_{\alpha})_{\alpha} $ is a bounded net, then $ A $ is already
biflat, that is, there exists a bounded $A$-bimodule morphism
$\rho:A\rightarrow (A\otimes_{p}A)^{**}$ such that
$\pi^{**}_{A}\circ\rho (a)=a$ for each $a\in A$. To see this, by Remark \ref{rem wot} the bounded net $
(\rho_{\alpha})_{\alpha} $ has a
 cluster point $ \rho\in B(A, (A\otimes_{p}A)^{**}) $ with respect to the weak* operator topology.
Note that
$$  b\cdot \rho(a)=w^{\ast}-\lim
b\cdot\rho_{\alpha}(a)=w^{\ast}-\lim\rho_{\alpha}(ba)=\rho(ba)$$ and
$$
 \rho(a)\cdot b=w^{\ast}-\lim\rho_{\alpha}(a)\cdot b=w^{\ast}-\lim\rho_{\alpha}(ab)=\rho(ab)\quad (a,b\in
 A).$$ Thus $\rho $ is a bounded $A$-bimodule morphism.
Moreover since $\pi^{**}_{A}$ is a $w^{*}$-continuous map, we have $
\pi^{**}_{A}\circ\rho(a)=w^{\ast}-\lim\pi_{A}^{**}\circ\rho_{\alpha}(a)=a
$ for every $a\in A.$ It follows that $A$ is a biflat Banach
algebra.
\end{Remark}
Combining Proposition \ref{F property} with Theorem \ref{jons imply
amen}, we have the following corollary:
\begin{cor}
Let $A$ be a Banach algebra with an identity. Then $A$ is amenable
if and only if $A$ has property $\digamma$.
\end{cor}
Let $ G $ be a locally compact group. A linear subspace $S^{1}(G)$ of $L^{1}(G)$ is said to be a Segal
algebra on $G$ if it satisfies the following conditions
\begin{enumerate}
\item [(i)] $S^{1}(G)$ is  dense    in $L^{1}(G)$,
\item [(ii)]  $S^{1}(G)$ with a norm $||\cdot||_{S^{1}(G)}$ is
a Banach space and $|| f||_{L^{1}(G)}\leq|| f||_{S^{1}(G)}$ for
every $f\in S^{1}(G)$,
\item [(iii)]  $ S^{1}(G) $ is left translation invariant (i.e. $L_{y}f\in S^{1}(G)$ for every $f\in S^{1}(G)$ and $y\in G$) and the map $y\mapsto L_{y} f$ from $G$ into $S^{1}(G)$ is continuous, where
$L_{y}f(x)=f(y^{-1}x)$,
\item [(iv)] $||L_{y}f||_{S^{1}(G)}=||f||_{S^{1}(G)}$ for every $f\in
S^{1}(G)$ and $y\in G$,
\end{enumerate}
for more information about Segal algebras, see \cite{rei}.

The group $G$ is said to have small invariant neighborhoods, denoted by $SIN$-group if in  every neighborhood of the identity there exists a compact invariant neighborhood of the identity.
\begin{Theorem}
Let $ G $ be an $SIN$-group. Then the following statements are equivalent:
\begin{enumerate}
\item [(i)] G is an amenable group,
\item [(ii)] $S^{1}(G)$ is Johnson pseudo-contractible,
\item [(iii)] $S^{1}(G)$ has property $\digamma.$
\end{enumerate}
\end{Theorem}
\begin{proof}
(i)$\Rightarrow$(ii) Since $ G $ is amenable, $ L^{1}(G) $ has a
virtual diagonal, say $ M $. On the other hand since $G$ is $SIN$
group  by \cite[theorem 1]{kot}  $S^{1}(G)$ has a central approximate
identity say $ (e_{\alpha}) $. Recall that $ (e_{\alpha}) $ is $
L^{1} $-norm bounded \cite[\S  8, proposition 1]{rei}. Set $
M_{\alpha}=e_{\alpha}\cdot M\cdot e_{\alpha} $. Then $ M_{\alpha}\in (S^{1}(G)\otimes_{p}S^{1}(G))^{\ast\ast}$.
Moreover for $ f\in S^{1}(G) $ we have
\begin{equation*}
M_{\alpha}\cdot f=f\cdot M_{\alpha},
\end{equation*}
and
\begin{equation*}
\begin{split}
\norm{\pi^{\ast\ast}(M_{\alpha})\cdot f-f}_{S^{1}(G)}\leq
&\norm{e_{\alpha}\cdot\pi^{\ast\ast}(M)\cdot e_{\alpha}\cdot
f-e_{\alpha}\ast f\ast e_{\alpha}}_{S^{1}(G)}
\\&+\norm{e_{\alpha}\ast f\ast e_{\alpha}-e_{\alpha}\ast f}_{S^{1}(G)}+\norm{e_{\alpha}\ast f-f}_{S^{1}(G)}
\\ \leq &\norm{\pi^{\ast\ast}(M)\cdot(e_{\alpha}\ast f\ast e_{\alpha})-e_{\alpha}\ast f\ast e_{\alpha}}_{S^{1}(G)}
\\&+\norm{e_{\alpha}}_{1}\norm{ f\ast e_{\alpha}-f}_{S^{1}(G)}+\norm{e_{\alpha}\ast f-f}_{S^{1}(G)}\longrightarrow 0,
\end{split}
\end{equation*}
where $ \cdot $ is the module action and $ \ast $ is the convolution product.

(ii)$\Rightarrow$(iii) This implication holds by Proposition \ref{F property}.

(iii)$\Rightarrow$(i) Since $S^{1}(G)$ has a left approximate
identity, using similar argument as in the proof of
\cite[Theorem 3.9]{sah-app-bipro}  one can show that if $S^{1}(G)$
has property $\digamma$, then $S^{1}(G)$ is left $\phi$-amenable.
Now applying \cite[Corollary 3.4]{alagh} we see that $G$ is
amenable.
\end{proof}
In fact the case (iii)$\Rightarrow$(i) holds for every locally compact group $ G $.
\begin{cor}
Let $G$ be a locally compact group. If $S^{1}(G)$ has property
$\digamma$, then $G$ is amenable.
\end{cor}
\begin{Remark}
There exists a Banach algebra with property $\digamma$ which
is not Johnson pseudo-contractible. To see this let $A=\left[\begin{array}{cc} 0&\mathbb{C}\\
0&\mathbb{C}\\
\end{array}
\right]$. Define $\rho:A\rightarrow A\otimes_{p}A$ by
$\rho(a)=a\otimes a_{0}$, where $a_{0}=\left[\begin{array}{cc} 0&1\\
0&1\\
\end{array}
\right]$. Clearly $\rho$ is a bounded $A$-bimodule morphism and
$\pi_{A}\circ\rho(a)=a$ for every $a\in A$. It follows that $A$ has
property $\digamma.$ But if $A$ is Johnson pseudo-contractible, then
by \cite[Corollary 2.7]{sah1} $A$ is pseudo-amenable. Thus $A$ has
an approximate identity say $(e_{\alpha})$. So there exists
$a_{\alpha}$ and $b_{\alpha}$ in $\mathbb{C}$ such that
$e_{\alpha}=\left[\begin{array}{cc} 0&a_{\alpha}\\
0&b_{\alpha}\\
\end{array}
\right]$. It follows that $$\left[\begin{array}{cc} 0&0\\
0&0\\
\end{array}
\right]=\left[\begin{array}{cc} 0&a_{\alpha}\\
0&b_{\alpha}\\
\end{array}
\right]\left[\begin{array}{cc} 0&1\\
0&0\\
\end{array}
\right]\rightarrow \left[\begin{array}{cc} 0&1\\
0&0\\
\end{array}
\right],$$ which is impossible.
\end{Remark}
\begin{lemma}
Let $A$ be a Banach algebra with property $\digamma$. Then $A^{2}$
is dense in $A.$
\end{lemma}
\begin{proof}
Since $A$ has property $\digamma$, there exists a net of
$A$-bimodule morphisms $(\rho_{\alpha})\subseteq B(A,(A\otimes_{p}A)^{**})$ such that
$\pi_{A}^{**}\circ\rho_{\alpha}(a)\rightarrow a$ for each $a\in A.$
Suppose by contradiction that $\overline{A^{2}}\neq A$. By the Hahn-Banach
theorem there exists a non-zero element $f\in A^{*}$ such that
$f|_{A^{2}}=0.$ Fix $a\in A$ such that $f(a)\neq 0.$ Let
$(x^{\alpha}_{\beta})_{\beta}$ be a net in $A\otimes_{p}A$ such that
$x^{\alpha}_{\beta}\xrightarrow{w^{*}}\rho_{\alpha}(a).$ It is worth noting explicitly that $\pi_{A}^{**}(x^{\alpha}_{\beta})\in A^2$. Since
$\pi^{**}_{A}$ is $w^{*}$-continuous map, we have
$\pi_{A}^{**}(x^{\alpha}_{\beta})\xrightarrow{w^{*}}\pi_{A}^{**}\circ\rho_{\alpha}(a)$.
So
$\pi_{A}^{**}(x^{\alpha}_{\beta})(f)\rightarrow\pi_{A}^{**}\circ\rho_{\alpha}(a)(f)$.
Also since $\pi_{A}^{**}\circ\rho_{\alpha}(a)\rightarrow a$, we have
$\pi_{A}^{**}\circ\rho_{\alpha}(a)\xrightarrow{w^{*}} a$, then
$\pi_{A}^{**}\circ\rho_{\alpha}(a)(f)\rightarrow f(a)$, that is,
$0=\lim_{\alpha}\lim_{\beta}\pi_{A}^{**}(x^{\alpha}_{\beta})(f)=f(a)$
which is a contradiction.
\end{proof}

\section{Johnson pseudo-contractibility and nilpotent ideals}\label{sec3}
A subspace $ E $ of a normed space $X$ is called
complemented in $X$, if  there exists a bounded operator
$P:X\rightarrow E$ such that $P(x)=x$ for every $x\in
E.$ 

\begin{lemma}\label{cr}
	Suppose that $ X $ is a normed space and $ E $ is a complemented subspace of $ X $. Let $ I: E \longrightarrow X  $ be the inclusion map. Then the map
	$ Id_{Y}\otimes I:Y\otimes_{p}E\longrightarrow Y\otimes_{p}X  $
	has closed range for any normed space $ Y $, where
	$ Id_{Y} $ is the identity map of $ Y $.
\end{lemma}
\begin{proof}
	Suppose that $y\in \overline{\mathrm{Im}(Id_{Y}\otimes I)}$. Then there exists a sequence  $ (x_{n})\subseteq
	Y\otimes_{p} E $ such that $ Id_{Y}\otimes
	I(x_{n})\longrightarrow y $. For every $ n $ we may assume that
	$x_{n}=\sum^{\infty}_{i=1}a_{i}^{n} \otimes b_{i}^{n} $ for some sequences $(a^{n}_{i})$ in $Y$ and $(b^{n}_{i})$ in $E$. Since  $ E $ is a
	complemented subspace of $ X $, there is a bounded operator $
	P:X\longrightarrow E $ such that $
	P(b_{i}^{n})=b_{i}^{n}
	$ for all $  i $ and $ n $. Hence for every $ n $,
	\begin{equation*}
	\begin{split}
	x_{n}-(Id_{Y}\otimes P)\circ(Id_{Y}\otimes
	I)(x_{n})&=\sum\limits_{i}a_{i}^{n} \otimes
	b_{i}^{n}-\sum\limits_{i}a_{i}^{n} \otimes P(b_{i}^{n})
	\\&=\sum\limits_{i}a_{i}^{n} \otimes (b_{i}^{n}-P(b_{i}^{n}))=0.
	\end{split}
	\end{equation*}
	Since  $
	Id_{Y}\otimes P  $ is a bounded map, we have $x_{n}= (Id_{Y}\otimes P)\circ(Id_{Y}\otimes I)(x_{n})\longrightarrow (Id_{Y}\otimes P)(y) $.
	It means that $(x_{n})$ converges to $x=(Id_{Y}\otimes P)(y)$ in
	$Y\otimes_{p}E.$
	So we have
	$$Id_{Y}\otimes I(x)=Id_{Y}\otimes I(\lim_{n}x_{n})=\lim_{n}Id_{Y}\otimes I(x_{n})=y,$$
	which means that $y\in\mathrm{Im}(Id_{Y}\otimes I)$.
	Thus $Id_{Y}\otimes I$ has closed range.
\end{proof}
\begin{Theorem}\label{thm33}
	Suppose that $A$ is a Johnson pseudo-contractible Banach algebra and
	$N$ is a closed  complemented ideal of $A$ and also $E$
	is a closed ideal of $A$ satisfying $E\subseteq N$ and $EN=\{0\}$.
	If $M$ is any subset of $A$ with $ME\subseteq
	\overline{EA} $ $ (EM\subseteq \overline{AE}) $, then
	$ME=\{0\} $ $ (EM=\{0\}) $, respectively.
\end{Theorem}
\begin{proof}
	We modify the argument of \cite[Lemma 2]{zhang}.
	Put $I:N\rightarrow A$ for the inclusion map and also put
	$q:A\rightarrow \frac{A}{E}$ for the quotient map. Let
	$Id_{N},Id_{A}$ and $Id_{\frac{A}{E}}$ be the identity operator on
	$N,A$ and $\frac{A}{E}$, respectively. Let
	$p:\frac{A}{E}\otimes_{p}N\rightarrow N$ be the map given by
	$p((a+E)\otimes n)=an$ for every $a\in A$ and $n\in N.$ Since
	$EN=\{0\}$ one can see that $p$ is well-defined.
	
	We go towards a contradiction and suppose that $ME\neq \{0\}$. So
	there exist  $m\in M$ and $e\in E$ such that $ me\neq 0$ and $me\in
	ME\subseteq \overline{EA}$, thus there exist some sequences
	$(a_{n})$ in $A$
	and $(e_{n})$ in $ E$ such that $e_{n}a_{n}\rightarrow me.$
	Since $A$ is Johnson pseudo-contractible, by Proposition \ref{F
		property} $A$ has property $\digamma$. So there exists a net
	$(\rho_{\alpha})_{\alpha}\subseteq B(A,(A\otimes_{p}A)^{**}) $ of $A$-bimodule morphisms such that
	$\pi^{**}_{A}\circ\rho_{\alpha}(a)-a\rightarrow 0,$ for every $a\in
	A.$  Since $me\neq 0$ and
	$\pi^{**}_{A}\circ\rho_{\alpha}(me)\rightarrow me\neq 0$, we can
	choose $\alpha$
	such that $\pi^{**}_{A}\circ\rho_{\alpha}(me)=\pi^{**}_{A}\circ\rho_{\alpha}(m)\cdot e\neq 0.$
	For $b\in N$, let
	$L_{b}$  $(R_{b})$ be the map of left (right)
	multiplication by $b,$ respectively.
	Define $$d_{\alpha}=(q\otimes R_{e})^{**}\circ \rho_{\alpha}(m)\in
	\bigg(\frac{A}{E}\otimes_{p}N\bigg)^{**}.$$ Consider
	\begin{equation*}
	\begin{split}
	p^{**}(d_{\alpha})=p^{**}((q\otimes R_{e})^{**}\circ
	\rho_{\alpha}(m))=[p\circ(q\otimes R_{e})]^{**}\circ
	\rho_{\alpha}(m) &=(R_{e}\circ\pi_{A})^{**}\circ\rho_{\alpha}(m)
	\\&=R^{**}_{e}\circ\pi^{**}_{A}\circ\rho_{\alpha}(m)
	\\ &=\pi_{A}^{**}\circ\rho_{\alpha}(m)\cdot e\neq 0,
	\end{split}
	\end{equation*}
	which implies that for a sufficiently large $\alpha$ we have $d_{\alpha}\neq
	0.$  Since, $q\circ I\circ L_{e_{n}}=0.$ we have
	\begin{equation*}
	\begin{split}
	(Id_{\frac{A}{E}}\otimes I)^{**}(d_{\alpha})&=[(Id_{\frac{A}{E}}\otimes I)\circ (q\otimes Id_{N})\circ (Id_{A}\otimes R_{e})]^{**}\circ\rho_{\alpha}(m)\\
	&=[(q\otimes Id_{A})\circ (Id_{A}\otimes I)\circ (Id_{A}\otimes
	R_{e})]^{**}\circ\rho_{\alpha}(m)
	\\&=(q\otimes Id_{A})^{**}\circ \rho_{\alpha}(m)\cdot e
	\\ &=(q\otimes Id_{A})^{**}\circ \rho_{\alpha}(me)
	\\ &=(q\otimes Id_{A})^{**}\circ \rho_{\alpha}(\lim_{n} e_{n}a_{n})
	\\ &=\lim_{n}(q\otimes Id_{A})^{**}\circ \rho_{\alpha}(e_{n}a_{n})
	\\ &=\lim_{n}[(q\otimes Id_{A})\circ (I\otimes Id_{A})\circ (L_{e_{n}}\otimes Id_{A})]^{**}\circ\rho_{\alpha}(a_{n})
	\\ &=\lim_{n}[(q\circ I\circ L_{e_{n}})\otimes Id_{A}]^{**}\circ\rho_{\alpha}(a_{n})=0.
	\end{split}
	\end{equation*}
	Since $N$
	is a closed  ideal which is complemented in $A$, so by \cite[Lemma 1]{zhang}  the map $Id_{\frac{A}{E}}\otimes I $
	is injective. Also by the previous lemma $Id_{\frac{A}{E}}\otimes I $ has closed range. Then by   \cite[A.3.48]{auto-dale} the map $(Id_{\frac{A}{E}}\otimes I)^{**}$ becomes injective.
	Thus the equation $(Id_{\frac{A}{E}}\otimes I)^{**}(d_{\alpha})=0$
	implies that $d_{\alpha}=0$ for each $\alpha$ which is impossible.
\end{proof}
\begin{Theorem}
	Let $A$ be a Johnson pseudo-contractible Banach algebra. Then $A$
	has  no non-zero closed nilpotent ideal which is 
	complemented in $A.$
\end{Theorem}
\begin{proof}
	We go towards a contradiction and suppose that we have a non-zero
	closed nilpotent ideal $N$. Let $n\in\mathbb{N}$ such that
	$N^{n}\neq {0}$ and $N^{n+1}={0}$. Let $E$ be the closure of the linear span of $N^{n}$. Obviously $ E $ is a non-zero closed ideal in $ A $. Moreover $E\subseteq N$ and $ EN=0 $. On the other hand since  $A$ is a
	Johnson pseudo-contractible Banach algebra, by \cite[Proposition 2.6]{sah1} $A$ is
	pseudo-amenable. Therefore $A$ has an approximate identity. So we have $\overline{AE}=\overline{EA}=E$. Since $AE\subseteq
	E=\overline{EA}$, Theorem \ref{thm33} implies that $AE=\{0\}$. Thus
	$E=\overline{AE}=\{0\}$ which is a contradiction.
\end{proof}
Since every finite dimensional subspace of a Banach algebra is
complemented we have the following corollary:
\begin{cor}
	Every Johnson pseudo-contractible Banach algebra has no non-zero
	finite dimensional nilpotent ideal.
\end{cor}

It is well-known that every closed subspace of a Hilbert space is
complemented, and up to isomorphism the only Banach spaces with this property are Hilbert spaces.

Also in a commutative Banach
algebra, the ideal generated by a nilpotent element is  nilpotent. So we
have the following corollary:
\begin{cor}
	Let $A$ be a  Johnson pseudo-contractible Banach algebra whose underlying space is a Hilbert
	space. Then $A$ has  no non-zero nilpotent closed ideal. Moreover if
	$A$ is commutative, then $A$ has  no non-trivial nilpotent element.
\end{cor}

{\bf Acknowledgments.} {The authors are  grateful to the referees for useful comments which improved the manuscript and for pointing out a number of misprints.}

\providecommand{\bysame}{\leavevmode\hbox to3em{\hrulefill}\thinspace}
\providecommand{\MR}{\relax\ifhmode\unskip\space\fi MR }
\providecommand{\MRhref}[2]{%
  \href{http://www.ams.org/mathscinet-getitem?mr=#1}{#2}
}
\providecommand{\href}[2]{#2}

\end{document}